\documentclass[a4paper,12pt]{amsart}

\usepackage{euscript,eufrak,verbatim}
\usepackage[usenames]{color}

\newtheorem{theorem}{Theorem}[section]
\newtheorem{lemma}[theorem]{Lemma}
\newtheorem{proposition}[theorem]{Proposition}
\newtheorem{corollary}[theorem]{Corollary}

\begin{document}

\title[Almost everywhere convergence of ergodic series]{Almost everywhere convergence of ergodic series}

\date{}

\author
{Aihua Fan}
\email{ai-hua.fan@u-picardie.fr}

\address
{LAMFA, UMR 7352 CNRS, University of Picardie,
33 rue Saint Leu, 80039 Amiens, France}

\begin{abstract} We consider ergodic series of the form
$\sum_{n=0}^\infty a_n f(T^n x)$ where  $f$ is an integrable function with zero mean value with respect to
a $T$-invariant measure $\mu$. Under certain conditions on the dynamical system
$T$, the invariant measure $\mu$ and the function $f$, we prove that
the series converges $\mu$-almost everywhere if and only if $\sum_{n=0}^\infty |a_n|^2<\infty$, and that in this case
the sum of the convergent series is exponentially integrable and satisfies a Khintchine type inequality.  We also prove that the system $\{f\circ T^n\}$ is a Riesz system if and only if the spectral
measure of $f$ is absolutely continuous with respect to the Lebesgue measure and the Radon-Nikodym derivative is bounded from above as well as from below by a constant. We check the conditions  for Gibbs measures $\mu$
relative to hyperbolic dynamics $T$ and for H\"{o}lder functions $f$.
An application is given to the study of differentiability of the Weierstrass type functions $\sum_{n=0}^\infty a_n f(3^n x)$.

\end{abstract}

\maketitle

\section{Introduction } Let $(\Omega, \mathcal{A}, \mu, T)$ be a measure-preserving dynamical
system. By an ergodic series we mean a series of the form
\begin{equation}\label{ES}
   \sum_{n=0}^\infty a_n f_n(T^n x).
\end{equation}
We assume that $f_n$ are integrable and $\mathbb{E} f_n=0$, and that $(a_n)$ a sequence of numbers.
We are interested in the almost everywhere (a.e.) convergence of the series.
We say that $\{f_n \circ T^n\}$ is a {\em convergence system} if
$$
   \sum_{n=0}^\infty |a_n|^2<\infty \Rightarrow \sum_{n=0}^\infty a_n f_n(T^nx)
   \ \mbox{converges} \ a.e.
$$
If, furthermore, the inverse implication is also true, we say that $\{f_n \circ T^n\}$ is an
{\em exact convergence system}.
In this paper, we try 
to find convergence systems and exact convergence systems of the form $\{f_n \circ T^n\}$.

The following negative result is due to
Kakutani and Petersen \cite{KP1981} (see also \cite{PetersenBook}, p.94 and p.99).
Suppose that $(a_n)$ is a decreasing sequence of positive numbers such that  $n a_n = O(1)$ and
$\sum_{n=0}^\infty a_n = \infty$. For any non-atomic  $T$-ergodic measure $\mu$, there is a function $f\in L^\infty(\mu)$ with
$\mathbb{E} f = 0$ such that the series $\sum_{n=0}^\infty a_n f( T^n x)$
diverges a.e.  This shows that in general
the condition $\sum_{n=0}^\infty |a_n|^2<\infty$ is not sufficient for the a.e. convergence of the ergodic series
(\ref{ES}) and we need
additional conditions on the system $(\Omega, \mathcal{A}, \mu, T)$  or on the functions $f_n$ or on both.

   In many cases the sequence $\{f_n\circ T^n\}$ has rapid decay of correlation, which implies
   that it is a quasi-orthogonal system. But even if $\{f_n\circ T^n\}$ is orthogonal,
   the condition $\sum_{n=0}^\infty |a_n|^2<\infty$
    is  not sufficient either. Actually, the situation for orthogonal and quasi-orthogonal series is well understood  through the following two results. The first one is due to Rademacher-Menshov (see \cite{Alexits}, p. 80 for orthogonal series), which is generalized by Kac-Salem-Zygmund \cite{KSZ} to quasi orthogonal series. The second one is due  to Tantori (see \cite{Alexits}, p. 88). If $$
         \sum_{n=0}^\infty a_n^2 \log^2 n <\infty,
   $$
   any quasi orthogonal series $\sum_{n=0}^\infty a_n X_n$ converges a.e. If
   $$
         a_n \downarrow 0, \quad \sum_{n=0}^\infty a_n^2 \log^2 n =\infty,
   $$
   then there exists an orthonormal system of functions $\{\Phi_n\}$ on the interval
   $[0, 1]$ depending on $\{a_n\}$ such that the series  $\sum_{n=0}^\infty a_n \Phi_n(x)$
   diverges a.e. with respect to  the Lebesgue measure.
   For other references on orthogonal or quasi-orthogonal series, see \cite{Alexits,  Weber}.

   Paszkiewicz \cite{P2010} gave a complete characterization of sequences $\{a_n\}$
for which $\sum a_n \Phi_n$ converges a.e. for {\em any} orthonormal sequence $\{\Phi_n\}$ in {\em any} $L^2$-space. When we {\em fix}  a  sequence $\{\Phi_n\}$, it is a different problem to find conditions on $\{a_n\}$
 for  $\sum a_n \Phi_n$ converges a.e.  The famous Lusin problem belongs to this category of problems and it
 treats the case $\Phi_n(x) = e^{i n x}$ for which Carleson \cite{Carleson} proved that  $\sum |a_n|^2<\infty$ is a sufficient condition
 for $\sum a_n e^{in x}$ converges a.e. with respect to the Lebesgue measure. In our paper,
 we will fix up some 'good' sequence $\Phi_n = f \circ T^n$ in the setting of measure-preserving dynamical systems so that $\sum a_n f\circ T^n$ converges a.e. whence $\sum |a_n|^2 <\infty$.

   To state our result, we need the notion of Riesz system.
   We say that $\{f_n \circ T^n\}$ is a {\em Riesz system} if the inequalities
   $$
      C^{-1} \sum |a_n|^2 \le \left\|\sum a_n f_n(T^n x) \right\|_{L^2(\mu)}^2
         \le C \sum |a_n|^2$$
   hold for some constant $C >1$ and for all finite sequences $(a_n)$. Recall
   that the inequality at the right hand side  means that $\{f_n \circ T^n\}$ is {\em quasi orthogonal} (see \cite{KSZ}). In this paper, we will prove the following theorem.

   \begin{theorem} \label{main} Assume that $(X, \mathcal{B}, T, \mu)$ is an ergodic measure-preserving dynamical system.
      Let $f \in L^1(\mu)$. Suppose
      \\
      \ \indent {\rm (H1)}\ \ $
      \lim_{n\to \infty} \mathbb{E} (f|T^{-n}\mathcal{B})=0$; \\
      \ \indent {\rm (H2)} \ $
          \sum_{i=0}^\infty \|\mathbb{E}(f|T^{-i}\mathcal{B}) - \mathbb{E}(f|T^{-(i+1)}\mathcal{B})\|_{L^\infty(\mu)} <\infty
      $.\\
      Then for any complex sequence $(a_n)\subset \mathbb{C}$ such that $\sum_{n=0}^\infty |a_n|^2<\infty$,
      the ergodic series $\sum_{n=0}^\infty a_n f(T^n x)$ converges a.e. In this case,
      we have the following Khithchine inequality
      \begin{equation}\label{Khintchine}
         \left\| \sum_{n=0}^\infty a_n f(T^n x) \right\|_{L^p(\mu)}
         \le C(p, f) \sqrt{\sum_{n=0}^\infty |a_n|^2}, \quad (p\ge 1)
      \end{equation}
      where $C(p, f)>0$ is a constant  independent of $(a_n)$. If furthermore $\{f \circ T^n\}$ is a Riesz system,
       then $\sum_{n=0}^\infty |a_n|^2=\infty$ implies the almost everywhere
         divergence of the
      series $\sum_{n=0}^\infty a_n f(T^n x)$.
    \end{theorem}

   The condition (H1) is satisfied by all $f\in L^1(\mu)$ with $\mathbb{E} f =0$ when $\mu$
   is not only ergodic but also exact (see \cite{Walters}, p. 115 for the exactness).  As we shall see,
   both conditions (H1) and (H2) are satisfied when $\mu$
   is a Gibbs measure and $f$ is rather regular for many hyperbolic dynamical systems (Section 4).

   We have mentioned the negative result of Kakutani-Petersen concerning the almost everywhere convergence of ergodic series. In particular, the series $\sum_{k=1}^\infty \frac{f(T^k x)}{k}$ diverges almost everywhere for some $f\in L^\infty$. However, Cotlar \cite{ Cotlar} had proved a positive result concerning the ergodic Hilbert
   transform:
       $$
            H_n f(x) : = {\sum_{k=-n}^n}' \frac{f(T^k x)}{k} = \sum_{k=1}^n \frac{f(T^k x) - f(T^{-k} x)}{k}
       $$
   where $T$ is an invertible measure-preserving transform. Cotlar's theorem states that the limit $\lim_{n\to \infty}
   H_nf(x)$ does exist almost everywhere for every integrable function $f$. These negative and positive results show that the almost
   everywhere convergence of the ergodic series (\ref{ES}) is a subtle problem.
    We point out that our result is actually proved for $\sum_{n=0} a_n f_n(T^n x)$ for a sequence of functions
    $\{f_n\}$ (Theorem \ref{main2}).

 In \cite{Fan_Riesz} and \cite{Peyriere}, the authors studied the almost everywhere convergence with respect to Riesz products $\prod_{n=1}^\infty (1 + {\rm Re} a_n e^{i\lambda_n x})$ of lacunary trigonometric series
 $\sum_{n=1}^\infty c_n (e^{i\lambda_n x} - \overline{a}_n/2)$. It is found that $\sum_{n=1}^\infty |c_n|^2<\infty$ is a necessary and sufficient condition. The method used in \cite{Fan_Riesz} works
 for all compact abelian groups.

     Consider $X_n=f(T^n x)$ as random variables. The study of the random series (\ref{ES}) is tightly related to the weighted SLLN (Strong Law of Large Numbers)
\begin{equation}\label{SLLN}
    \lim_{n\to \infty} \frac{\sum_{k=0}^nw_k X_k}{\sum_{k=0}^n w_n} = 0 \quad a.s.
\end{equation}
where $(w_n)$ is a sequence of positive numbers such that $\lim_{n\to\infty}\sum_{k=0}^n w_n =+\infty$.
By the Kronecker lemma (\cite{Shiryayev}, p.390), the almost sure convergence of (\ref{ES}) implies that of (\ref{SLLN}), when $a_n =w_n/W_n$ with
$W_n = w_0+w_1+\cdots +w_n$.
Concerning i.i.d. sequences $(X_n)$, a pioneering work on the weighted SLLN is due to Jamison, Orey and Pruitt \cite{JOP}. There
were many subsequent works done in this direction for independent variables, but few for dependent variables
\cite{CL}.
In the dynamical case where $X_n=f \circ T^n$, the study of the series (\ref{ES})
may be considered as a study on the rate of convergence in the Birkhoff ergodic theorem. As shown in Krengel \cite{Krengel} where a short section
(pp. 14-15) is devoted to the discussion on the subject, no general estimate holds for
the ergodic sum $\sum_{k=0}^n f\circ T^k(x)$. As pointed out by Kachurovskii \cite{K} (p. 654),
a systematic study is then to find parameters on which certain characteristics of the rate of convergence depend. Theorem \ref{main} provides us a rather general parameter,  which is composed of (H1) and (H2), ensuring that almost everywhere
$$
   \sum_{k=0}^n f\circ T^k(x) = O\left(\sqrt{n\log_1 n \log_2 n \cdots \log_{m-1} n \log_m^{1+\epsilon}n}\right)
$$
as $n\to \infty$,
where $\log_m n$ is the iterated logarithmic function, inductively defined by
$\log_1 n = \log n$ and  $\log_m n = \log \log_{m-1} n$ ($n$ being large enough).

    In Theorem \ref{main} we need the hypothesis that $\{f \circ T^n\}$ is a Riesz system.
    Let $f \in L^2([0, 1])$ which defines an odd function on $[-1, 1]$ and is then extended to a $2$-periodic
    function on $\mathbb{R}$. Hedenmalm, Lindqvist and Seip \cite{HLS} characterized those functions $\varphi$
    such that the dilated functions $\{\varphi(nx)\}_{n\ge 1}$ is a Riesz basis of $L^2([0, 1])$.
    Such a function $\varphi$ provides us a Riesz system $\{\varphi(3^n x)\}_{n\ge 0}$ in $L^2([0, 1])$.
    In general,
    whether $\{f_n \circ T^n\}$ is a Riesz system is a delicate matter and it will be worthy of study.

    For the system $\{f\circ T^n\}$ to be a Riesz system in $L^2(\mu)$, we get the following criterion.
    For any measure $\sigma$ on $[-\pi, \pi]$, we define
    $$
    \sigma^* = \frac{\sigma + \check{\sigma}}{2}
    $$
    where $\check{\sigma}(A)=\sigma(-A)$ for all Borel sets $A$ in $[-\pi, \pi]$.
    In the following theorem, $\lambda$ refers to the Lebesgue measure on $[-\pi, \pi]$.

\begin{theorem}\label{ThmRiesz} Assume that $(X, \mathcal{B}, T, \mu)$ is a  measure-preserving dynamical system.
      Let $ f\in L^2(\mu)$ with spectral measure $\sigma_f$. \\
      \indent \mbox{\rm (i)} \ \ We have $\sigma_f\ll \lambda$ and $\frac{d\sigma_f}{d\lambda}\le A^2$ for some constant $A>0$ iff the inequality
      \begin{equation}\label{QO1}
            \left\|\sum a_n f\circ T^n\right\|_2 \le A \sqrt{\sum |a_n|^2}
      \end{equation}
      holds for all finite complex sequences $(a_n)$.\\
       \indent \mbox{\rm (ii)} \ We have $\lambda \ll \sigma_f$ and $\frac{d\lambda}{d\sigma_f}\le B^{-2} $
        for some constant $B>0$ iff the inequality
      \begin{equation}\label{QO-1}
            \left\|\sum a_n f\circ T^n\right\|_2 \ge B \sqrt{\sum |a_n|^2}
      \end{equation}
      holds for all finite complex sequences $(a_n)$.
      \\
      \indent \mbox{\rm (iii)} \ \ Assume that we consider the inequalities (\ref{QO1}) and (\ref{QO-1})
with {\rm real} sequences $(a_n)$.   Then we have the same results (i) and (ii)  but with $\sigma_f$ replaced by $\sigma_f^*$.
\end{theorem}

  The proof of Theorem \ref{main} is based on a maximal inequality due to
  Erd\"os-Steckin-Gaposkin (Lemma \ref{MaxIneq}), and on the Khintchine inequality (\ref{Khintchine})
  with $p=4$ which is proved by studying the concentration of the series
  (Lemma \ref{Azuma}).

  As an application,
  Theorem \ref{main}  will be applied to study the
  differentiability of the Weierstrass type functions
  \begin{equation} \label{F1}
      F(x) = \sum_{n=0}^\infty a_n 3^{-n} f(3^n x) \qquad (x\in \mathbb{T})
\end{equation}
  where $f \in C^{1+\delta}(\mathbb{T})$ with some $\delta >0$ such that 
  $\{f'(3^n x)\}$ is a Riesz system,  and $(a_n)$ is a sequence of numbers
such that
 \begin{equation}\label{Cond-a}
\lim_{n\to\infty} a_n =0,\quad
\sum_{n=0}^\infty a_n \ {\rm diverges}, \quad
\sum_{n=0}^\infty |a_n -a_{n+1}|<\infty.
\end{equation}
The following phase transition is proved (Theorem \ref{Thm-diff}).

\begin{theorem} \label{Thm-diff0} Consider the  function $F$
defined by (\ref{F1}), where $f\in C^{1+\delta}(\mathbb{T})$ for some $\delta >0$ such that $\{f'(3^n)\}_{n\ge 0}$ is a Riesz system in $L^2(\mathbb{T})$ and $(a_n)$ is a sequence of numbers satisfying the condition (\ref{Cond-a}).
  We have the following dichotomy:\\
\indent {\rm (a)} \  If $\sum_{n=0}^\infty |a_n|^2 =\infty$, then
$F$ is almost everywhere non-differentiable and admits a set of full Hausdorff dimension of differentiable
points.\\
\indent {\rm (b)} \  If $\sum_{n=0}^\infty |a_n|^2 <\infty$, then
$F$ is almost everywhere differentiable and admits a set of full Hausdorff dimension of non-differentiable
points.
\end{theorem}

 Notice that the condition $\lim_{n\to 0} a_n=0$ is necessary. See \cite{Katznelson} for the fact that
 if $a_n\not\to 0$, then $\sum a_n 3^{-n} e^{2\pi i 3^n x}$ is nowhere differentiable.
 See \cite{BBR} for information on the classical Weierstrass functions.
 \medskip

 The rest of the  paper is organized as follows. Section 2 is devoted to the proof of Theorem \ref{main}.
 Section 3 characterizes Riesz systems of the form $\{f \circ T^n\}$. Section 4 discusses the
 conditions (H1) and (H2) in Theorem \ref{main} in the context of Gibbs measures of hyperbolic dynamical systems.
 The differentiability of  Weierstrass type functions is studied in the last section.

\setcounter{equation}{0}

\section{Almost everywhere convergence}
In this section, we prove a result which is a little more general than  Theorem \ref{main}, our main result
announced in the introduction.

\subsection{Some preparative lemmas}

The material in this subsection is 
standard. For the sake of completeness, we will give the proofs of Lemma \ref{MaxIneq} and Lemma \ref{Azuma}
in an appendix after the proof of Theorem \ref{main}.
Our proof of Theorem \ref{main} is based the following elementary but powerful maximal inequality.
It was first used by Erd\"{o}s in the setting of trigonometric series.
Steckin stated it without giving proof and it was first proved by Gaposhkin \cite{Gaposhkin}.
Moricz made a generalization in \cite{Moricz1976}. 

\begin{lemma} \label{MaxIneq} Let $\{\xi_n\}_{n\ge 0}$ be a sequence of random variables
and $\{a_n\}_{n\ge 0}$  a sequence of numbers.
For $0\le p\le q$, denote $S_{p, q} =\sum_{n=p}^q a_n\xi_n$.
Suppose there exist constants $\beta >2$
and $C >0$ such that for all positive integers $p$ and $q$ with $p\le q$ we have
$$
     \|S_{p, q}\|_\beta  \le C \sqrt{\sum_{n=p}^q |a_n|^2}.
$$
Then there exists $C' >0$ depending only on $\beta $ and $C$ such that
$$
    \left\|\max_{p\le k\le q}|S_{p, k}|\right\|_\beta  \le C' \sqrt{\sum_{n=p}^q |a_n|^2}
$$
holds for all $0\le p\le q$.
We can take $C' = \frac{C}{1- 2^{1/\beta -1/2}}$.
\end{lemma}

Now let us present a martingale decomposition in the dynamical setting and prove an Azuma type
inequality.
Let $(X, \mathcal{B}, T, \mu)$ be a measure preserving dynamical system.
    Let
    $$
          \mathcal{B}^i = T^{-i} \mathcal{B}, \qquad \mathcal{B}^\infty = \lim_{i\to \infty} T^{-i} \mathcal{B}.
    $$
    For $\phi \in L^1(\mu)$ and $i\ge 0$, define
    $$
           d_i(\phi) = \mathbb{E}(\phi|\mathcal{B}^{i}) - \mathbb{E}(\phi|\mathcal{B}^{i+1}).
    $$
    Notice that $d_i(\phi)=0$ if $\phi$ is $\mathcal{B}^{i+1}$-measurable.
    We have the following decomposition for $\phi$.

    \begin{lemma}\label{decomposition} Let $\phi \in L^1(\mu)$. Suppose that  $\lim_{n\to \infty}\mathbb{E} (\phi| \mathcal{B}^n)=0$ a.e. Then
we have the
decomposition
$$
      \phi = \sum_{i=0}^\infty d_i(\phi)
$$
where the series converges almost surely.
\end{lemma}

In general, the reverse martingale $\mathbb{E}(\phi|\mathcal{B}^{n})$ converges to a limit, which is
$\mathcal{B}^\infty$-measurable but may be not a constant, even if $\mu$ is ergodic. For example,
it is the case of irrational rotation.
 However if $\mu$ is exact, i.e. trivial on $\mathcal{B}^\infty$, then the limit is zero
 for any $\phi$ such that
$\mathbb{E} \phi =0$. Note that the condition $\lim_{n\to \infty}\mathbb{E} (\phi| \mathcal{B}^n)=0$
implies $\mathbb{E} \phi =0$.
\medskip

 The following is an inequality of Azuma type for $\sum_{i=0}^\infty d_i(\phi)$.

\begin{lemma} \label{Azuma} Let $\phi \in L^1(\mu)$ be a real function. Suppose
 $$\lim_{n\to \infty} \mathbb{E} (\phi|\mathcal{B}^n) =0, \qquad
\sum_{i=0}^\infty \|d_i(\phi)\|_\infty^2<\infty.
$$
 Then $\phi$ is subgaussian in the sense that
for any real number $\lambda$,
$$
    \mathbb{E} \exp(\lambda \phi) \le \exp \left(\frac{\lambda^2}{2} \sum_{i=0}^\infty \|d_i(\phi)\|_\infty^2\right).
$$
Consequently, for any positive number
$p\ge 1$, we have
$$
     \left\|\phi\right\|_{L^p(\mu)} \le \sqrt{2} \Gamma(p/2)^{1/p}  \sqrt{\sum_{i=0}^\infty \|d_i(\phi)\|_\infty^2}.
$$
\end{lemma}

   We will apply the above Azuma inequality to
   $\phi = \sum_{k=0}^n a_k f \circ T^k$. The following well known lemma will be useful
   in the computation.

   \begin{lemma}\label{Lem*} For any $k\le i$, we have $\mathbb{E} (\phi \circ T^k| \mathcal{B}^i) = \mathbb{E} (\phi | \mathcal{B}^{i-k}) \circ T^{k}$.
    \end{lemma}

   The following well known lemma shows that $\mathbb{E} (\phi|T^{-k}\mathcal{B})$
   can be expressed by the Perron-Frobenius operator $\widehat{T}_\mu: L^1(\mu)\to L^1(\mu)$ which is  defined by
    $$
          \int g \cdot \widehat{T}_\mu \phi d\mu = \int g\circ T \cdot \phi d\mu
    $$
    for $\phi\in L^1(\mu)$ and $ g\in L^\infty(\mu)$. The operator $\widehat{T}_\mu$ will be simply denoted $\widehat{T}$.

   \begin{lemma}\label{transfer}
    $\mathbb{E} (\phi|T^{-k}\mathcal{B})= \widehat{T}^kf \circ T^k$.
   \end{lemma}


\subsection{Proof of Theorem \ref{main}}
Now let us prove Theorem \ref{main}. Actually we prove a little more.
 Recall that   for $f\in L^1(\mu)$ and $i\ge 0$,
    $$
           d_i(f) = \mathbb{E}(f|\mathcal{B}^{i}) - \mathbb{E}(f|\mathcal{B}^{i+1}).
    $$

    \begin{theorem} \label{main2} Assume that $(X, \mathcal{B}, T, \mu)$ is an ergodic measure-preserving dynamical system.
      Let $(f_n) \subset L^1(\mu)$. Suppose
      \\
      \ \indent {\rm (H1)}\ \ $\forall k\ge 0$,
      $\lim_{n\to \infty} \mathbb{E} (f_k|T^{-n}\mathcal{B})=0$; \\
      \ \indent {\rm (H2)} \ There is a function $\Delta : \mathbb{N}\to \mathbb{R}^+$ with $\sum_{i=0}^\infty \Delta (i)<\infty$ such that for all $0\le k\le i$
      $$
           \|d_i(f_k\circ T^k)\|_{L^\infty(\mu)} \le
          \Delta(i-k).
      $$
      Then for any complex sequence $(a_n)\subset \mathbb{C}$ such that $\sum_{n=0}^\infty |a_n|^2<\infty$,
      the ergodic series $\sum_{n=0}^\infty a_n f_n(T^n x)$ converges almost everywhere and
      we have the following Khithchine inequality
      $$
         \left\| \sum_{n=0}^\infty a_n f_n(T^n x) \right\|_{L^p(\mu)}
         \le C_p \sqrt{\sum_{n=0}^\infty |a_n|^2}, \quad (p\ge 1)
      $$
      where $C_p>0$ is a constant  independent of $(a_n)$. If furthermore $\{f_n \circ T^n\}$ is a Riesz system,
       then $\sum_{n=0}^\infty |a_n|^2=\infty$ implies the almost everywhere
         divergence of the
      series $\sum_{n=0}^\infty a_n f_n(T^n x)$.
    \end{theorem}

\begin{proof}
We are going to apply the Azuma inequality (Lemma \ref{Azuma}) to
   $$\phi :=S_n := \sum_{k=0}^{n} a_k f_k\circ T^k \qquad (n\ge 0)$$

   First, using Lemma \ref{Lem*} and the hypothesis (H1) we check that
   $$
      \lim_{i\to \infty} \mathbb{E}(\phi|\mathcal{B}^i)
      = \sum_{k=0}^{n} a_k \lim_{i\to \infty} \mathbb{E}(f_k|\mathcal{B}^{i-k})\circ T^k =0.
   $$
   Secondly we check the other condition in Lemma \ref{Azuma}.
   For $i\le n$, we write
   $$
        S_n = S_{i} + R_{i}, \quad R_{i} =\sum_{i+1\le k<n} a_k f_k\circ T^k.
   $$
   For $i> n$, we take $R_{i}=0$ so that $S_{i} = S_n$.
   Since $R_{i}$ is $\mathcal{B}^{i+1}$-measurable, a fortiori $\mathcal{B}^i$-measurable, we have
   $$
        d_i(S_n) = \mathbb{E}(S_{i}|\mathcal{B}^i) - \mathbb{E}(S_{i}|\mathcal{B}^{i+1})
        = \sum_{k=0}^{i\wedge n} a_k [\mathbb{E}(f_k\circ T^k|\mathcal{B}^i) - \mathbb{E}(f_k\circ T^k|\mathcal{B}^{i+1})].
   $$
   Therefore, by Lemma \ref{Lem*} and (H2) we get
   $$
       \|d_i(S_n)\|_\infty \le \sum_{k=0}^{i\wedge n} |a_k| \Delta(i-k).
   $$
   Now, by the Cauchy-Schwarz inequality, we have
   $$
       \|d_i(S_n)\|_\infty^2 
       \le  \sum_{k=0}^{i\wedge n} \Delta(i-k)  \sum_{k=0}^{i\wedge n} |a_k|^2 \Delta(i-k)\le \Delta^*\sum_{k=0}^{i\wedge n} |a_k|^2 \Delta(i-k)
   $$
   where $\Delta^* = \sum_{i=0}^\infty \Delta(i)$. Then take sum over $i$ to get
   \begin{eqnarray*}
    \sum_{i=0}^\infty \|d_i(S_n)\|_\infty^2
    &\le & \Delta^* \sum_{i=0}^\infty \sum_{k=0}^{{i\wedge n}} |a_k|^2 \Delta(i-k)\\
    &=& \Delta^* \sum_{k=0}^{n} |a_k|^2 \sum_{i \ge k} \Delta(i-k)\\
    &=&
    \Delta^{*2}
    \sum_{k=0}^{n} |a_k|^2<\infty.
    \end{eqnarray*}
Thus we can apply Lemma \ref{Azuma} to obtain the Khintichine inequality
    $$
       \left\|\sum_{k=0}^{n} a_k f_k\circ T^k\right\|_{L^p(\mu)}
       \le \Delta^* C_p \sqrt{\sum_{k=0}^{n} |a_k|^2}.
    $$

This Khintchine inequality with any $p>2$ and the maximal inequality (Lemma \ref{MaxIneq})
immediately imply the result on the convergence, by a standard argument (see \cite{Shiryayev}, p. 251-252).

Prove now the assertion on the divergence.
The hypothesis of Riesz system implies that there exists a constant $C>0$ such that
$$
    \left\|\sum_{k=0}^{n} a_k f_k\circ T^k \right\|_2^2 \ge C^2 \sum_{k=0}^{n} |a_n|^2.
$$
By the Paley-Zygmund inequality (\cite{Kahane1985}, p. 8) and the already proved Khintchine inequality ($p=4$), for any $0<\lambda<1$ we have
$$
   \mu\left( |S_n|^2 \ge \lambda \mathbb{E} S_n^2\right) \ge (1- \lambda)^2 \frac{(\mathbb{E} |S_n|^2)^2}{\mathbb{E}|S_n|^4} \ge  (1- \lambda)^2 \frac{C^4}{C_4^4} >0.
$$
As $\mathbb{E} S_n^2 \ge C^2 \sum_{k=0}^{n} |a_n|^2$ tends to the infinity, it follows that
the event $\{\sup_{n\ge 0} |S_n|=+\infty\}$  has a strictly positive measure.
 However the event is invariant and the measure $\mu$ is ergodic. 
Thus
$\sup_{n\ge 0} |S_n|=+\infty$ $\mu$-a.e.
  \end{proof}

  Assume $f_k= f$ for all $k$. By Lemma \ref{Lem*}, we can write
   $$
       \mathbb{E}(f\circ T^k|\mathcal{B}^i) - \mathbb{E}(f\circ T^k|\mathcal{B}^{i+1})
       = \mathbb{E}(f|\mathcal{B}^{i-k})\circ T^k - \mathbb{E}(f|\mathcal{B}^{i+1-k})\circ T^k.
   $$
   In this case we can take
   $$
         \Delta(i) = \left\|\mathbb{E}(f|\mathcal{B}^{i}) - \mathbb{E}(f|\mathcal{B}^{i+1})\right\|_\infty.
   $$
   Thus Theorem\ref{main} follows immediately.

   \subsection{Appendix} We give here the proofs of Lemma \ref{MaxIneq} and Lemma \ref{Azuma}.

{\bf Proof of Lemma \ref{MaxIneq}.}  For $0\le p\le q$, let
$$
\sigma_{p, q}^2 = C^2\sum_{k=p}^q a_n^2, \qquad M_{p, q}=\max_{p\le k\le q}|S_{p, k}|.
$$
What we have to prove is the following proposition:
$$
  \\|S_{p, q}\|_\beta \le \sigma_{p, q}  \ \ \Rightarrow  \ \
  \left\|M_{p, q}\right\|_\beta \le C_\beta \sigma_{p, q}
$$
where $C_\beta >1$ is some constant. Let $d=q-p$. We prove this implication by induction on $d$.
If $d=0$, it is trivial for any $C_\beta \ge 1$. Assume $d\ge 1$.

Suppose the implication holds when $0\le q-p<d$. Now assume $q-p=d$. Let $q'\in [p, q]$ be the least
integer such that $\sigma_{p, q'}^2$ exceeds the half of $\sigma_{p, q}^2$. Then
\begin{equation}\label{M1}
    \sigma_{p, q'-1}^2 \le \frac{1}{2}\sigma_{p, q}^2 <\sigma_{p, q'}^2
\end{equation}
where, by convention,  $\sigma_{p, q'-1}^2$ is $0$  when $q'=p$. Consequently,
\begin{equation}\label{M2}
    \sigma_{q'+1, q}^2 = \sigma_{p, q}^2 - \sigma_{p, q'}^2 \le \frac{1}{2}\sigma_{p, q}^2
\end{equation}
where, by convention, $\sigma_{q'+1, q}^2$ is $0$  when $q'=q$. For $q'\le k \le q$, we have
$$
     |S_{p, k}| \le |S_{p, q'}| + |S_{q'+1, k}|\le |S_{p, q'}| + M_{q'+1, q}.
$$
For $p\le k<q'$, we have $|S_{p, k}|\le M_{p, q'-1}$. So for all $p\le k \le q$, we have
$$
      |S_{p, k}| \le  |S_{p, q'}| + (M_{q'+1, q}^\beta + M_{p, q'-1}^\beta)^{1/\beta}
$$
so that
$$
      |M_{p, q}| \le  |S_{p, q'}| + (M_{q'+1, q}^\beta + M_{p, q'-1}^\beta)^{1/\beta}
$$
By the Minkowski inequality in $L^\beta$, we get
$$
 \|M_{p, q}\|_\beta \le  \|S_{p, q'}\|_\beta + (\mathbb{E}M_{q'+1, q}^\beta + \mathbb{E}M_{p, q'-1}^\beta)^{1/\beta}.
$$
By the induction hypothesis and the inequalities (\ref{M1}) and (\ref{M2}), both $\mathbb{E}M_{q'+1, q}^\beta$
and $\mathbb{E}M_{p, q'-1}^\beta$ is bounded by $(C_\beta \sigma_{p, q}/\sqrt{2})^\beta$.
Thus
$$
 \|M_{p, q}\|_\beta \le  (1 + C_\beta 2^{\frac{1}{\beta}-\frac{1}{2}}) \sigma_{p, q}.
$$
Therefore $\|M_{p, q}\|_\beta \le C_\beta \sigma_{p, q}$ if
$$
1 + C_\beta 2^{\frac{1}{\beta}-\frac{1}{2}}\le C_\beta.
$$
We can choose $C_\beta = \frac{1}{1-  2^{\frac{1}{\beta}-\frac{1}{2}}}$.
$\Box$
\medskip


{\bf Proof of Lemma \ref{Azuma}.} The convexity of the exponential function implies
$$
      e^{u x} \le \frac{1+u}{2} e^x + \frac{1- u}{2} e^{-x}
      \qquad (x \in \mathbb{R}, -1\le u \le 1).
$$
For simplicity, write $\mathbb{E}^i(\cdot) = \mathbb{E}^i(\cdot|\mathcal{B}^i)$ and $d_i = d_i(\phi)$. Let $\lambda \in \mathbb{R}$. Since $\mathbb{E}^{i+1}d_i=0$, applying the above inequality to $u=\frac{d_i}{\|d_i\|_\infty}$
and $x = \lambda \|d_i\|_\infty $  we get
\begin{equation}\label{Estimate-di}
    \mathbb{E}^{i+1} e^{\lambda d_i}
    \le \cosh (\lambda \|d_i\|_\infty) \le e^{\lambda^2/2 \|d_i\|_\infty^2}.
\end{equation}
Then for any integer $n$, by the estimate (\ref{Estimate-di}) with $i=0$ we get
$$
   \mathbb{E} e^{\lambda \sum_{i=0}^n d_i} =  \mathbb{E} \left(
   e^{\lambda\sum_{i=1}^n d_i} \mathbb{E}^1 e^{ d_0}\right)
   \le e^{\lambda^2/2 \|d_0\|_\infty^2} \mathbb{E} e^{\lambda\sum_{i=1}^n d_i}.
   $$
By induction, we get
$$
    \mathbb{E} e^{\lambda \sum_{i=0}^n d_i} \le e^{\lambda^2/2\sum_{i=0}^{n-1} \|d_i(\phi)\|_\infty^2}
$$
Let $n\to \infty$, we obtain the subgaussian property, by Fatou lemma.

The $L^p$-norm of $\phi$ is estimated as follows. By the subgaussian property and the Markov inequality, we have
   $$
      \mathbb{E} |\phi|^p = p\int_0^\infty t^{p-1} \mu(|\phi|\ge t) d t
        \le 2 p\int_0^\infty t^{p-1} e^{\lambda^2\sigma^2/2 -\lambda t} dt
   $$
   where $\sigma^2 = \sum_{i=0}^\infty \|d_i\|_\infty^2$. Take $\lambda = t/\sigma^2$
   (the minimiser of $\lambda \mapsto \lambda^2\sigma^2/2 - \lambda t$). Then
   $$
       \mathbb{E} |\phi|^p \le 2 p \int_0^\infty t^{p-1} e^{-t^2/(2\sigma^2)} dt
       = \sigma^{p}\cdot 2 p \int_0^\infty s^{p-1} e^{-s^2/2} d s
       = 2^{p/2}\Gamma(p/2) \sigma^p.
   $$
   $\Box$
   \medskip

\medskip

\setcounter{equation}{0}

\section{Riesz system $\{f\circ T^n\}_{n\ge 0}$}

We present here a proof of Theorem \ref{ThmRiesz} which gives a necessary and sufficient condition for $\{f\circ T^n\}$ to be a Riesz system.
The condition is expressed by the spectral measure of $f$. For the notion
of spectral measure and the useful spectral lemma, we  refer to \cite{Krengel} (pp. 94-95).

\medskip

{\bf Proof of Theorem \ref{ThmRiesz}. }   We first make a remark. If $\phi$ is non negative continuous function on the circle $\mathbb{R}/2\pi\mathbb{Z}$ which is identified with $[-\pi, \pi]$, then there exists a sequence of complex polynomials $Q_n \in \mathbb{C}[X]$ such that
\begin{equation}\label{approx}
              \phi(t) = \lim_{n\to \infty} |Q_n(e^{it})|^2
\end{equation}
where the limit is uniform on $t \in [-\pi, \pi]$. In fact, consider the convolution
$\phi_n = \phi * K_n$ where $K_n$ is the $n$-th F\'ejer kernel. By the F\'ejer theorem,
$\phi_n$ converges uniformly to $\phi$.
The function $\phi_n$
is a real trigonometric polynomial of order $n$. Then by the Riesz lemma (see \cite{PS}, Exercise 40 in Chapter VI),
there exists a polynomial $Q_n\in \mathbb{C}[X]$
such that $\phi_n(t) = |Q_n(e^{it})|^2$. Thus (\ref{approx}) is proved.

Furthermore, if $\phi$ is non negative continuous function on the circle $\mathbb{R}/2\pi\mathbb{Z}$ satisfying $\phi(t)= (-t)$, then there exists a sequence of real polynomials $Q_n \in \mathbb{R}[X]$ such that
(\ref{approx}) holds. In fact, since both $\phi$ and $K_n$ are even, so is $\phi_n$. Then $\phi_n$
is in the space spanned by $1, \cos t, \cos 2t, \cdots,$  $ \cos n t$ or equivalently spanned by
$1, \cos t, \cos^2t, \cdots,$ $ \cos^n t $.  In other words, $\phi_n(t) = P_n(\cos t)$
for some polynomial $P_n \in \mathbb{R}[X]$.    Now we apply
 the Riesz lemma in its variant form (see  \cite{PS}, Exercise 41 in Chapter VI. also see \cite{Daubechies}, p. 172).

Let us start the proof for (i).
By the spectral lemma (\cite{Krengel}, p. 94), the inequality (\ref{QO1}) can be written as
\begin{equation*}\label{QO2}
    \forall Q \in \mathbb{C}[X], \quad
            \int_{-\pi}^\pi|Q(e^{i t})|^2 d\sigma_f (t) \le A^2 \int_{-\pi}^\pi |Q(e^{ i t})|^2 d\lambda(t).
      \end{equation*}
      By the above remark (see the equality (\ref{approx})), this inequality is equivalent to
\begin{equation*}\label{QO2}
    \forall \phi \in C([-\pi, \pi]) \ {\rm with}\ \phi\ge 0, \quad
            \int_{-\pi}^\pi \phi(t) d\sigma_f (t) \le A^2 \int_{-\pi}^\pi \phi(t)  d\lambda(t).
      \end{equation*}
      This is equivalent to what we have to prove for (i). Similar proof holds for (ii).

Now let us prove (iii).
By the spectral lemma (\cite{Krengel}, p. 94), the inequality (\ref{QO1}) can be written as
\begin{equation}\label{QO2}
    \forall Q \in \mathbb{R}[X], \quad
            \int_{-\pi}^\pi|Q(e^{i t})|^2 d\sigma_f (t) \le A^2 \int_{-\pi}^\pi |Q(e^{ i t})|^2 d\lambda(t).
      \end{equation}
Observe that $|Q(e^{i t})|^2$ is even. Then
$$
   \int_{-\pi}^\pi|Q(e^{i t})|^2 d\sigma_f^* (t)
   = \int_{-\pi}^\pi\frac{|Q(e^{i t})|^2  +|Q(e^{-i t})|^2}{2}d\sigma_f (t)
   =\int_{-\pi}^\pi|Q(e^{i t})|^2 d\sigma_f (t)
$$
So, the inequality (\ref{QO2}) becomes
\begin{equation}\label{QO3}
    \forall Q \in \mathbb{R}[X], \quad
            \int_{-\pi}^\pi|Q(e^{i t})|^2 d\sigma_f^* (t) \le A^2 \int_{-\pi}^\pi |Q(e^{ i t})|^2 \lambda(t).
      \end{equation}
      Using the remark concerning even functions $\phi$,
we get immediately that  the inequality (\ref{QO3}) is equivalent to the following inequality:
for all $ 0\le \phi \in C([-\pi, \pi])$ with $\phi(-t)=\phi(t)$,
\begin{equation}\label{QO4}
            \int_{-\pi}^\pi \phi d\sigma_f^*  \le A^2 \int_{-\pi}^\pi \phi d\lambda.
      \end{equation}
      The inequality (\ref{QO4}) remains true even if $\phi$ is not event. In fact, decompose
      $\phi = \phi_e + \phi_o$ where $\phi_e$ is even  and $\phi_o$ is odd. Then
      $$
          \int_{-\pi}^\pi\phi d\sigma_f^*  = \int_{-\pi}^\pi \phi_e d\sigma_f^*, \quad
          \int_{-\pi}^\pi\phi d\lambda = \int_{-\pi}^\pi \phi_e d\lambda
      $$
      because $\int \phi_o d\sigma_f^* =\int \phi_o d\lambda=0$. Notice that $\phi_e \ge 0$.
    The  inequality (\ref{QO4}) with arbitrary non-negative $\phi$ is what we have to prove
      for (i) in the case of real sequences $(a_n)$. In the same way, we can prove (ii) in the case of real sequences $(a_n)$.
$\Box$

\setcounter{equation}{0}

\section{Gibbs measures in hyperbolic dynamical systems}
 In this section, we examine the conditions (H1) and (H2) for Gibbs measures in  hyperbolic dynamical systems.
 Let us consider a topological dynamical system $(X, T)$ where
   $X$ is a compact metric space with metric $d$ and $T: X\rightarrow X$  a
continuous map (we can relax the compactness of the space $X$ and the continuity of the transformation $T$).
For a given strictly positive continuous
 $\psi: X \rightarrow {\mathbb R}^+_*$, called
a {\em potential}, we define
the {\em Ruelle-Perron-Frobenius operator}
${\mathcal L}={\mathcal L}_{\psi}$, simply called Ruelle operator, by
$$
{\mathcal L}\phi (x) =\sum_{y\in T^{-1}(x)} \psi (y) \phi (y)
$$
where $\phi$ is in a suitable space of functions on $X$. The Ruelle theorem states that under suitable conditions, there is a positive number $\rho>0$ and a positive function $h>0$ and a probability measure
$\nu$ such that$$
       {\mathcal L} h = \rho h, \quad {\mathcal L}^* \nu = \rho \nu.
$$
In this case, we usually consider the normalized Ruelle operator $L$ corresponding to the
new potential $\psi(x)h(Tx)/(\rho h(x))$. The measure $\mu$ defined by $d\mu = hd\nu$
is a $T$-invariant measure, called the Gibbs measure associated to $\psi$ and denoted $\mu_\psi$. It can be checked that
for this Gibbs measure $\mu=\mu_\psi$, the Perron-Frobenius operator $\widehat{T}_\mu$
is equal to the Ruelle operator $L$.
 So, the conditions (H1) and (H2) in Theorem \ref{main} are satisfied  if
 \begin{equation}\label{L-cond}
     \lim_{n\to \infty} \|L^n f\|_\infty =0, \qquad \sum_{n=0}^\infty \|L^n f - L^{n+1} f \|_\infty <\infty
 \end{equation}
 which is weaker than
 \begin{equation}\label{L-cond-s}
     \sum_{n=0}^\infty \|L^n f\|_\infty <\infty.
 \end{equation}
 There are many works done on the decay of $\|L^n f\|_\infty$. Thus
 Theorem \ref{main} applies to a large class of Gibbs measures
and regular functions associated to hyperbolic dynamical systems, for which the condition
(\ref{L-cond-s}) is satisfied (see \cite{Bowen,FJ1,FJ2,PP, Sarig,Young1998, Young1999}).
 We will just recall some of them and check if (\ref{L-cond-s}) is satisfied. As an illustrating example,  we will make a detailed study on expanding endomorphisms on the torus $\mathbb{T}^d$.

  \subsection{Anosov systems}
     Let $T$ be an Anosov diffeomorphism on a closed compact smooth manifold $M$. Assume that $T$
     is topologically mixing. The following facts are true \cite{Bowen, R3}.
     There exists a Markov partition for $(M, T)$ and
     a symbolic dynamical system $(\Sigma_A, \sigma)$ and a map $\Pi: \Sigma_A \to M$ satisfying:\\
       \ \ \indent (1)   $\Pi$ is surjective and H\"{o}lder continuous;\\
       \ \ \indent (2) $\Pi$ is a semi-congugacy, i.e. $T\circ \Pi =\Pi \circ T$\\
       \ \ \indent (3) $\Pi$ is   finite-to-one;\\
        \ \ \indent (4)  For any ergodic measure of full support, $\Pi: M_0 \to \Sigma_A$
          is one-to-one for some $M_0 (\subset M)$ of full measure.

          The statistical study of the dynamics $(M, T)$ is thus converted to that
          of the symbolic dynamics $(\Sigma_A, \sigma)$. It is known that \cite{PP} if $\mu$ is a Gibbs measure associated to a H\"{o}lder continuous potential and if  $f$ is a H\"{o}lder
          continuous function such that $\int f d\mu=0$,  then $\|L^n f\|_\infty$ decays exponentially fast, so that $\{f\circ T^n\}$
           is a convergence system. We give some details in the following case of expanding systems.


 \subsection{Expanding systems}

We say that a dynamical system $T$ on a compact metric space $X$  is {\em locally expanding}
if there are constants $\lambda > 1$ and $b>0$
such that
$$
x,y \in X,\;\; d(x,y)\leq b \Rightarrow d(Tx, Ty) \geq \lambda d(x,y).
$$
It is said to be {\em mixing} if for
any non-empty open set $U$ of $X$, there is an integer $n>0$ such that
$T^{n}(U)=X$.

For any $n \geq 0$, we define a new
metric $d_n$ on $X$, called {\em $n$-Bowen metric}, as
$$
d_n(x, y) = \max_{0\leq j\leq n} d(T^j x, T^j y).
$$
The {\em $n$-Bowen ball} centered at $x \in X$ of  radius $r>0$ is
denoted by $B_n(x, r)$. The $0$-Bowen metric is just the original metric
$d$ on $X$. The $0$-Bowen ball $B_0(x, r)$ will be denoted by $B(x, r)$.

Let ${\mathcal C}={\mathcal C}(X, {\mathbb R})$ be
the space of all continuous functions $\phi: X\rightarrow {\mathbb R}$ with
the
supremum norm
$$
||\phi||_\infty =\max_{x\in X} | \phi (x)|.
$$
For a right continuous and increasing function
$\omega: {\mathbb R}^+ \rightarrow {\mathbb R}^+$
with $\omega (0) = 0$ (called  {\em  modulus of continuity}),
we define
 ${\mathcal H}^{\omega}={\mathcal H}^{\omega}(X, {\mathbb R})$ to be
  the space of all $\omega$-H\"{o}lder functions
$\phi\in {\mathcal C}$, those satisfying
$$
[\phi]_\omega = \sup_{x, y\in X, 0<d(x,y)\leq a} \frac{|\phi (x) -\phi (y)|}
{\omega(d(x,y))} <\infty
$$
(We choose a number $0<a\le b$).
For $\phi \in {\mathcal H}^{\omega}$, we define the norm
$$
    \|\phi\|_{\omega} = [\phi]_\omega + \|\phi\|_\infty.
$$
It is easy to see that $( {\mathcal H}^{\omega}, \|\cdot \|_\omega)$ is a Banach space.

A modulus of continuity $\omega (t)$ is said to satisfy
 {\em Dini condition} if
$$
\int_0^1 \frac{\omega (t)}{t} d t < \infty.
$$
For such a Dini function $\omega$, define
$$
\tilde{\omega}(t)=\sum_{n=1}^{\infty} \omega(\lambda^ {-n}t).
$$
It is easy that $\tilde{\omega}$ is also a
modulus of continuity.

Let ${\mathcal M}$ be the dual space of ${\mathcal C}$ and
let ${\mathcal L}^*: {\mathcal M}\to {\mathcal M}$ be
the adjoint operator of ${\mathcal L}: {\mathcal C}\to {\mathcal C}$.
For any measure $\nu\in {\mathcal M}$ and any function
$\phi\in {\mathcal C}$, we use
$
\langle \nu, \phi\rangle
$
to denote the integral of $\phi$ with respect to $\nu$.

Let us recall the Ruelle theorem  proved in \cite{FJ1}.

\begin{theorem}[Ruelle Theorem]
Suppose that $\omega$ is a Dini modulus of continuity
and $\psi \in {\mathcal H}^\omega$.
The following statements hold:\\
\indent {\rm (1)} \
 There exist a strictly positive number $\rho$ and a
strictly positive function $h \in {\mathcal H}^{\tilde{\omega}} $
such that ${\mathcal L} h = \rho h$.\\
\indent {\rm (2)} \  There exists a unique probability measure $\nu=\nu_{\psi} \in
{\mathcal M}$
such that ${\mathcal L}^*\nu = \rho \nu$.
\\
\indent {\rm (3)} \  For sufficiently small $r_0> 0$,
there is a constant $C=C(r_0)>0$ such that
$$
C^{-1} \leq \frac{\nu\big( B_n(x,r_0)\big)}
                          {\rho^{-n}G_{n}(x)}
\leq C \qquad \hbox{(Gibbs property)}
$$
holds for all $x \in X$ and $n \geq 1$,
where $G_n(x) =\prod_{j=0}^{n-1} \psi(T^jx)$.
\\
\indent {\rm (4)} \  Take $h$ in (1) such that $\langle \nu,h \rangle =1$. Then for any
$\phi\in {\mathcal C}$,
$$
\lim_{n\to \infty} \|\rho^{-n} {\mathcal L}^{n} \phi- \langle \nu, \phi \rangle h \|_\infty =0.
$$
\end{theorem}

Take an eigenfunction $h$ (associated to the eigenvalue $\rho$)
such that $\langle \nu, h \rangle =1$. The measure $\mu$ such that $d\mu = h d\nu$
is called the {\em Gibbs measure} associated to the potential $\psi$. Let
$$
   \widetilde{\psi}(x) = \psi(x) \frac{h(x)}{\rho h(Tx)}.
$$
The potential $\widetilde{\psi}$ shares the same Gibbs measure as $\psi$.
If $L$ denotes the Ruelle operator associated to
$\widetilde{\psi}$, we have
$$
   L 1 = 1, \quad L^n \phi =\rho^n h \mathcal{L}^n (\phi h^{-1}).
$$
Without loss of generality, we will assume that $\mathcal{L}$ is normalized
in the sense that $\mathcal{L}1 = 1$. The following estimate on $\|L^n \phi\|_\infty$
is proved in \cite{FJ2}.

\vskip10pt
\begin{theorem}\label{FJ_Thm}
Make the same assumptions as in Theorem 1.
Let $\mu$ the Gibbs measure associated to $\psi$. Assume that $L1 = 1$ and $\int \phi d\mu=0$.
If both $\psi$ and $\phi$ are in $\mathcal{H}^\omega$ with $\omega (t) = 1/|\log t|^\alpha$
with $\alpha >1$. Then
$$
\|L^n\phi\|_\infty\leq C\frac{(\log n)^{\alpha}} {n^{\alpha -1}}
\qquad (n \geq 1).
$$
\end{theorem}

\begin{corollary}
If $\alpha >2$, $\{\phi \circ T^n\}$ is a convergent system in $L^2(\mu)$.
\end{corollary}
We can get a little better in the  special case of expanding endomorphisms on the torus, where $\alpha >1$ will be proved to be a sufficient condition. We will present two direct proofs without using Theorem \ref{FJ_Thm}.

\subsection{Expanding endomorphisms on $\mathbb{T}^d$}
Here we consider a special dynamical system $({\Bbb T}^d, A)$
where ${\Bbb T}^d = {\Bbb R}^d /{\Bbb Z}^d$ ($d\geq 1$) is the
$d$-dimensional torus and $A$ is an endomorphism on ${\Bbb T}^d $.
We suppose that $A$ is {\it expanding}, that is,  all of its
eigenvalues have absolute value strictly larger than $1$.
We take the Haar-Lebesgue  measure  $\mu = dx$ on ${\Bbb T}^d$.
The system $({\Bbb T}^d, A, \mu )$ is strong mixing. For an integral function
$f$ defined on ${\Bbb T}^d$ with $\int f d\mu =0$, we consider the
general lacunary series $$
\sum_{n=0}^\infty a_n f(A^n x).$$ The (classical) lacunary
series correspond to the case where $f$ is a group character of ${\Bbb T}^d$, i.e. an exponential function.
We are going to present two methods, one of which uses tilings and the other uses Fourier analysis.

We need some facts on affine tilings and
we refer to \cite{GM,LW} for the facts recalled below and for
further information. Given a measurable set $T \subset {\Bbb R}^d$,  $|T|$ denotes its Lebesgue measure.
Given two measurable sets $T$ and $S$,
the notation $T \simeq S$ means that $T$ and $S$ are equal
up to a set of  null Lebesgue measure.  An endomorphism of the torus is represented by an integral matrix.
Suppose $A$ is a $d \times d$ integral matrix which is {expanding},
that is,  all
of its eigenvalues $\lambda_i$ have $|\lambda_i|>1$. Denote
$\lambda = \inf |\lambda_i|$ and
$q = |{\rm det} A|$ ($q\geq 2$ and is an integer).

Take a digit set $D$ consisting of
representatives
  of
cosets in ${\Bbb Z}^d/A {\Bbb Z}^d$.
For each $\gamma \in D$, define the contraction map
$S_\gamma : {\Bbb R}^d \rightarrow {\Bbb R}^d$
by
     $$
       S_\gamma x = A^{-1} (x + \gamma).
   $$
There exists a unique compact set $T$
having the self-affinity
    $$
        T = \bigcup_{\gamma \in D} S_\gamma(T).
    $$
When $\gamma' \not= \gamma''$, $|S_{\gamma'}(T) \bigcap S_{\gamma''}(T)|=0$.  Therefore the self-affinity
implies
      \begin{equation}\label{Tiling1}
      \sum_{k \in D} 1_T(Ax -k) = 1_T(x)  \qquad {\rm a.e.}
   \end{equation}
It is also known that the compact set $T$ has  the tiling property
   \begin{equation}\label{Tiling2}
      \sum_{k \in {\Bbb Z}^d} 1_T(x -k) = 1 \qquad \qquad {\rm a.e.}
   \end{equation}
Since $T$ satisfies (\ref{Tiling1}) and (\ref{Tiling1}), we say it
generates an {\em integral self-affine tiling}.
The compact tile $T$ has the property $|T|=1$.
Here $T$ representing a tile is not to be confused with the meaning of dynamics in
the previous sections. Our dynamics is now represented by the matrix $A$.

The tiling property allows us to identify ${\Bbb T}^d$
with $T$ up to a null measure set. The self-affinity
allows us to decompose $T$ into $q$ disjoint (up to a
null measure set)  self-affine parts.

For a function $f$ defined on the torus,  its
modulus of continuity
is defined as
$$
    \Omega_f(\delta) = \sup_{|x-y|\leq \delta} |f(x) - f(y)|.
$$

\begin{theorem} If $f: {\Bbb T}^d \to \mathbb{C}$ is a Dini continuous function in the sense that
$\int_0^1 \Omega_f(t) \frac{dt}{t}<\infty$ such that $\int f(x) dx =0$, then
$\{f(A^n x)\}_{n\ge 0}$ is a convergence system in $L^2(\mathbb{T}^d)$.
\end{theorem}

\begin{proof} For our dynamics $(\mathbb{T}^d, \mathcal{B}(\mathbb{T}^d), A, \mu)$
where $\mu$ is the Lebesgue measure (the Gibbs measure associated to the constant potential $\psi\equiv 1$), the Ruelle operator is defined by
     $$
     L f(x) = \frac{1}{q}
          \sum_{\gamma \in D}
           f\left(A^{-1}(x + \gamma ) \right).
     $$
     By Lemma \ref{transfer}, we have
     $\mathbb{E} (\phi|A^{-k}\mathcal{B})= L^k\phi \circ A^k$. In order to apply
     Theorem \ref{main}, it suffices to show that
     \begin{equation}\label{A-estimate}
     \|L^k f\|_\infty \le C \Omega_f(\lambda^{-n})
\end{equation}
where $C>0$ is a constant and $\lambda$ is the least modulus of all eigenvalues of $A$,
because
$$
   \sum_{i=0}^\infty \left\| \mathbb{E} (f|A^{-i}\mathcal{B})\right\|_\infty
   =
   \sum_{i=0}^\infty \left\|L^i f\right\|_\infty
   \le   C'\int_0^1\frac{\Omega_f(t)}{t} dt <\infty.
$$

Let us prove (\ref{A-estimate}).
For $\gamma=(\gamma_1, \cdots, \gamma_n) \in D^n$, write
$$
     S_\gamma x = S_{\gamma_n}\circ S_{\gamma_2}\circ
     \cdots \circ S_{\gamma_1}x,
     \qquad
     T_\gamma = S_\gamma(T).
$$
Clearly
$$
     S_\gamma x  = A^{-n}x +A^{-n}\gamma_1 +\cdots + A^{-2}\gamma_{n-1}
          + A^{-1} \gamma_n.
$$
Denoting $b_\gamma =S_\gamma 0$, we get
$$
   L^n f (x) = \frac{1}{q^n }
                \sum_{\gamma \in D^n}
                f\left( A^{-n}x   + b_\gamma \right).
$$

For $\gamma \in D^n$,
write
$$
     f_\gamma = \frac{1}{|T_\gamma|}\int_{T_\gamma} f(x) d x.
$$
Note that $|T_\gamma| = q^{-n} |T|= q^{-n}$ and that
$$
    \sum_{\gamma \in D^n} f_\gamma =  q^{n} \int_T f
    = q^n \int_{{\Bbb T}^d} f=0.
$$
Then
\begin{eqnarray*}
       L^n f (x)
        =     \frac{1}{q^n} \sum_{\gamma \in D^n}
               f\left(
                       A^{-n} x + b_\gamma \right)
        =    \frac{1}{q^n} \sum_{\gamma \in D^n}
               \left[
               f\left(
                       A^{-n} x + b_\gamma \right)
                      - f_\gamma
               \right].
\end{eqnarray*}
Since $T_\gamma = A^{-n}T +b_\gamma$, we have immediately
$$
       |L^n f (x)|
       \leq  \frac{1}{q^n} \sum_{\gamma \in D^n} \Omega_f({\rm diam}\
       A^{-n}T )
          \leq  C \Omega_f(\lambda^{-n})
$$
where ${\rm diam } B$ denotes the diameter of a set $B$.
We used the fact that ${\rm diam} A^{-n} T \leq a \lambda^{-n}$
for some $a>0$ and the fact that
 $\Omega_f(2\delta) \leq 2 \Omega_f(\delta)$.
The estimate on $\|L^n f\|_\infty$ is thus proved.
\end{proof}

If $A$ is not expanding but hyperbolic, it belongs to the class of Anosov systems.
\medskip

The second methods uses the relation between the Ruelle operator and the
Fourier transform.

\begin{proposition}\label{P_TF}   Let $A^*$ be the transposed matrix of $A$. For $f \in L^1(\mathbb{T}^d)$, we have
\begin{equation}\label{Ruelle-Fourier}
    L^n f(x)
     =   \sum_{k \in {\Bbb Z}^d}
                       \hat{f}(A^{*n} k) e^{2 \pi i
                          \langle k, x\rangle
                                    }
         \qquad (\forall n \geq 1).
\end{equation}
\end{proposition}

\begin{proof}
It suffices to prove the expression (\ref{Ruelle-Fourier}) for $n=1$.
Take a digit set $D^*$ representing ${\Bbb Z}^d /A^*{\Bbb Z}^d $.
Assume that $0 \in D^*$.
Write
$$
   f(x) = \sum_{k \in {\Bbb Z}^d} \hat{f}(k) e^{2 \pi i \langle k, x\rangle}
        = \sum_{k \in {\Bbb Z}^d} \sum_{\beta \in D^*}
            \hat{f}(A^*k+\beta) e^{2 \pi i \langle A^*k+ \beta, x\rangle}.
$$
We have
$$
   f(A^{-1}(x+\gamma))
        = \sum_{k \in {\Bbb Z}^d} \sum_{\beta \in D^*}
            \hat{f}(A^* k+ \beta) e^{2 \pi i
                          \langle A^* k+\beta,\  A^{-1}(x+\gamma)\rangle}
$$
then
\begin{eqnarray*}
    L f(x)
    & = & \frac{1}{q} \sum_{k \in {\Bbb Z}^d}
                      \sum_{\beta \in D^*}
                      \sum_{\gamma \in D}
                      \hat{f}(A^* k+ \beta) e^{2 \pi i
                          \langle k+ A^{*-1}\beta, \ x+\gamma\rangle
                                    }  \\
    & = & \frac{1}{q} \sum_{k \in {\Bbb Z}^d}
                                e^{2 \pi i \langle k,\ x\rangle}
                      \sum_{\beta \in D^*}
                           \hat{f}(A^* k+ \beta)
                        e^{2 \pi i \langle A^{*-1}\beta,\ x\rangle}
                      \sum_{\gamma \in D}
                       e^{2 \pi i
                          \langle A^{*-1}\beta, \ \gamma\rangle
                                    }
\end{eqnarray*}
So, in order to get (\ref{Ruelle-Fourier}) with $n=1$,
 it suffices to note that
$$
    \frac{1}{q}   \sum_{\gamma \in D}
               e^{2 \pi i  \langle A^{*-1}\beta,\ \gamma\rangle}
               = \left\{ \begin{array}{ll}
                              1 & \quad {\rm if} \ \
                                    \beta =0  \ \ (\!\!\!\!\mod A^*{\Bbb Z}^d)\\
                              0 & \quad {\rm if} \  \
                                    \beta \not=0 \ \ (\!\!\!\!\mod A^*{\Bbb Z}^d).
                        \end{array}
               \right.
$$
In fact, suppose the above sum is not zero.
Since the group $D$
is a product of cyclic groups and
$m^{-1}\sum_{j=0}^{m-1} e^{2 \pi i j x} = 1 $ or $0$ for a real number $x$
according to  $x\in {\Bbb Z}$ or not, 
$\langle A^{*-1}\beta,\ \gamma'\rangle$ must an integer for any cyclic
factor group generator $\gamma'$. Then
$\langle A^{*-1}\beta,\ \gamma\rangle$  is an integer for
any $\gamma \in D$. Let $z \in {\Bbb Z}^d$. Write
$z = \gamma + A k$ with  $\gamma \in D$ and $k \in {\Bbb Z}^d$.
Then
$$
    \langle A^{*-1}\beta,\ z\rangle  =
      \langle A^{*-1}\beta,\ \gamma\rangle
      +  \langle \beta,\ k \rangle = 0    \quad{\rm (mod \ {\Bbb Z})}.
$$
It follows that $\beta = 0$ (mod $A^*{\Bbb Z}^d$).
\end{proof}

This proposition allows us to check the conditions (H1) and (H2) by using conditions on the Fourier
coefficients of $f$. For simplicity,
let us only consider  the dynamics $T: [0, 1)\to [0, 1)$
defined by $T x = q x$ $\!\!\mod 1$ where $q\ge 2$ is an integer.
For simplicity, we write $\mathbb{E}^k (\cdot) = \mathbb{E}(\cdot | T^{-k}\mathcal{B})$.
For an integrable function $f$, by Proposition \ref{P_TF} and Lemma \ref{Lem*} we have
$$
    \mathbb{E}^k f = \sum_{m \in \mathbb{Z}} \widehat{f} (q^k m) e^{2\pi i q^k m x}=\sum_{q^k | n } \widehat{f} (n) e^{2\pi i n x}.
$$
Since $q^{k+1}|n$ implies $q^{k}|n$, we have
$$
  d_k(f)= \mathbb{E}^k f - \mathbb{E}^{k+1} f =\sum_{q^k | n, q^{k+1} \not{\mid} \  n } \widehat{f} (n) e^{2\pi i n x}.
$$
Therefore if the Fourier series of $f$  converges absolutely and $\widehat{f}(0)=0$, we have $\lim_{k\to \infty}  \mathbb{E}^k f=0$ and
$$
   \sum_{k=0}^\infty \left\|\mathbb{E}^k f - \mathbb{E}^{k+1} f\right\|_\infty
   \le \sum_{k=0}^\infty \sum_{q^k | n, q^{k+1} \not | n } |\widehat{f} (n)|\le \|f\|_{A(\mathbb{T})}
   <\infty
$$
where $\|f\|_{A(\mathbb{T})} = \sum_{n\in \mathbb{Z}} |\widehat{f}(n)|$.
We conclude by the following statement.

 \begin{proposition}
 If $f$ is the sum of an absolutely convergent Fourier series such that $\widehat{f}(0)=0$,
then $\{f(q^n x)\}$ is a convergence system in $L^2([0, 1])$.
\end{proposition}

See \cite{Aistleitner2013, BW} for the study of general series $\sum a_k f(n_k x)$ with respect to Lebesgue measure.


\section{Differentiability of Weierstrass type functions}
In this last section, we give an application of Theorem \ref{main} to the study of the
differentiability of a class of generalized Weierstrass  functions.

Let $f \in C^1(\mathbb{T})$. Let $(a_n)$ be a sequence of numbers
such that $a_n \to 0$ as $n\to \infty$. We consider the continuous function
\begin{equation}\label{F}
      F(x) = \sum_{n=0}^\infty a_n 3^{-n} f(3^n x) \qquad (x\in \mathbb{T})
\end{equation}
and we would like to study its differentiability. If $f(x)=e^{2\pi i x}$ and $a_n =1$ for all $n$,
we get the famous Weierstrass function $\sum_{n=0}^\infty 3^{-n} e^{2\pi i 3^n x}$ which is nowhere differentiable (see \cite{Katznelson}, p. 106). For the modified
Weierstrass function $\sum_{n=1}^\infty a_n 3^{-n} e^{2\pi i 3^n x}$ to be differentiable somewhere, it is then necessary to assume that $a_n \to 0$.

In order to be able to apply our main theorem, we assume that $\{f'(3^n x)\}$ is a Riesz system in $L^2(\mathbb{T})$.
Such functions do exist. We identify $\mathbb{T}$ with $[0, 1)$. Hedenmalm, Lindqvist and Seip \cite{HLS} had characterized
all functions $\varphi$ such that $\{\varphi(nx)\}_{n\ge 1})$ is a Riesz basis of $L^2([0, 1])$.
So we can take $f$ such that $f'=\varphi$. Let us recall the criterion of Hedenmalm, Lindqvist and Seip. Since $\{\sqrt{2}\sin \pi n x\}_{n\ge 1}$ is an orthonormal basis of
$L^2([0, 1])$, any square integrable function $\varphi$ can be developed as follows
$$
         \varphi(x) = \sum_{n=1}^\infty c_n \sqrt{2}\sin \pi n x.
$$
Hedenmalm, Lindqvist and Seip proved that $\{\varphi(nx)\}_{n\ge 1}$ is a Riesz basis iff the Dirichlet
series $\sum_{n=1}^\infty c_n n^{-s}$ defines a function which is analytic and bounded away from zero and the infinity in the right half-plane $\mbox{Re}\, s >0$. In particular, it is the case when $c_n$
is totally multiplicative and $\sum_{p} |c_p|<\infty$ (summation over all primes). Thus the following are good examples
$$
    \varphi_\tau(x) = \sqrt{2} \sum_{n=1} \frac{\sin \pi n x}{n^\tau}\qquad (\tau >1).
$$
Another sufficient condition presented in \cite{HLS} is $\sum_{n=2}^\infty |c_n|<c_1=1$.
For example, $\varphi(x) = \sin \pi x + c_3 \sin 3\pi x$ with $|c_3|<1$ is a good example. But $c_3= -1$
produces a bad one. In the terminology of dynamical system, $\sin \pi x - \sin 3\pi x$
is a coboundary for the dynamics $x \mapsto 3x$ $\mod 1$.

The next proposition describes the differentiability of $F$ at a given point by the convergence of the formal
derivative series at that point. Thus the study of differentiability becomes the study of convergence.

\begin{proposition} \label{Diff-Cov} Let $F$ be the function defined by (\ref{F}), where $a_n\to 0$ and $f\in C^{1+\delta}(\mathbb{T})$ for some $\delta >0$.
Let $x \in \mathbb{T}$ be fixed. The function $F$  is differentiable at the point $x$ iff the series
    $
       \sum_{n=0}^\infty a_n f'(3^n x)
    $
    converges at $x$. In this case, we have
    \begin{equation}\label{F'}
         F'(x) = \sum_{n=0}^\infty  a_n f'(3^n x).
    \end{equation}
\end{proposition}
\begin{proof}
For any integer $N\ge 1$ and any non zero number $h$, we have
$$
   \frac{F(x+h) - F(x)}{h}
       = \sum_{n=1}^N a_n \frac{f(3^n (x+h))- f(3^n x)}{3^nh } \ + \ O\left(h^{-1}\sum_{n=N+1}^\infty a_n 3^{-n}\right).
$$
Notice that $a_N^* := \max_{n\ge N+1}|a_n|$ tends to zero as $N$ tends to the infinity
and that
$\sum_{n=N+1}^\infty a_n 3^{-n}\le 3^{-N}a_N^*$. Since $f'$ is $\delta$-H\"{o}lder,  by the mean value theorem we have
 $$
      f(b) - f(a) = f'(a) (b-a) + O(|b-a|^{1+\delta})
 $$
 for all $a$ and $b$, where the constant in $O(1)$ is independent of $a$ and $b$. Consequently
$$
    f(3^n (x+h))- f(3^n x)= f'(3^n x)3^n h + O(|3^n h|^{1+\delta}).
$$
Notice that $a_n$ is bounded and $\sum_{n=0}^N 3^{\delta n}$ is of size $3^{\delta N}$. Now we conclude that
$$
   \frac{F(x+h) - F(x)}{h}
       = \sum_{n=0}^N a_n f'(3^n x) + O\left(|3^N h|^\delta\right)
       +O\left(  a_N^* |3^N h|^{-1} \right).
$$
Observe that $\sqrt{a_n^*}3^{-n}$ decreases to zero. Then
for any fixed small $h$, we can choose an integer $N$ such that
\begin{equation}\label{DD}
         \sqrt{a_N^*}3^{-N}\le |h| \le  \sqrt{a_{N-1}^*}3^{-N+1}.
\end{equation}
In other words, $\sqrt{a_N^*}\le 3^N|h|\le  3 \sqrt{a_{N-1}^*}$.
We choose this $N$ so that
$$
    3^N|h|<  3 \sqrt{a_{N-1}^*} \to 0, \quad  a_N^* |3^N h|^{-1} \le \sqrt{a_N^*} \to 0.
$$
Thus we have proved the equality
$$
    \lim_{h\to 0}  \frac{F(x+h) - F(x)}{h}
       = \lim_{N\to \infty }\sum_{n=0}^N a_n f'(3^n x)
$$
where we understand that one limit exists iff the other limit exists. Notice that for any small $h$
we have determined an $N$ verifying (\ref{DD}). We should point out that for any $N$, we can find $h$
verifying (\ref{DD}).
\end{proof}

Now we investigate the differentiability of $F$ by studying the size of the set of differentiable
points of $F$, which is denoted by
$$
     D(F) = \{x\in \mathbb{T}: F'(x)\ \ \mbox{\rm exists}\}.
$$
The set of singular points of $F$ will be denoted by $S(F) = \mathbb{T}\setminus D(F)$.

The following theorem says that a class of functions $F$ which is defined below  is divided into two subclasses according
to $\sum_{n=0}^\infty |a_n|^2 <\infty$ or $=\infty$. Concerning the differentiability,  we will observe a "phase transition" from one subclass
to another. The proof  will be based not only our Theorem \ref{main}  but also on a
 result due to Fan and Schmeling \cite{FS} which states that there is a big set on which
Birkhoff sums are bounded (the Birkhoff ergodic theorem admits only an order $o(n)$). The proof also uses an argument of thermodynamical formalism. We refer to \cite{FS} for the thermodynamical formalism.

\begin{theorem} \label{Thm-diff} Let $f\in C^{1+\delta}(\mathbb{T})$ for some $\delta >0$ and that $\{f'(3^nx)\}_{n\ge 0}$ is a Riesz system in $L^2(\mathbb{T})$.
Let $(a_n)$ be a sequence of numbers such that $$
\lim_{n\to\infty} a_n =0,\quad
\sum_{n=0}^\infty  a_n \ \mbox{\rm diverges}, \quad
\sum_{n=0}^\infty |a_n -a_{n+1}|<\infty.$$
 Consider the continuous function $F$
defined by (\ref{F}). We have the following dichotomy:\\
\indent {\rm (a)} \  If $\sum_{n=0}^\infty |a_n|^2 =\infty$, then
$S(F)$ has full Lebesgue measure while the Hausdorff dimension of $D(F)$ is equal to $1$;\\
 \indent {\rm (b)} \  If $\sum_{n=0}^\infty |a_n|^2 <\infty$, then
$D(F)$ has full Lebesgue measure while the Hausdorff dimension of $S(F)$ is equal to $1$.
\end{theorem}

\begin{proof} First observe that $\int_0^1 f'(x) d x = f(1) -f(0) =0$ for $f$ is periodic.

Suppose $\sum_{n=0}^\infty |a_n|^2 =\infty$. By Theorem \ref{main}, $\{f'(3^nx)\}_{n\ge 0}$ is an exact convergence system and the series $\sum_{n=0}^\infty a_n f'(3^n x)$ diverges almost everywhere. This, together with Proposition \ref{Diff-Cov}, proves that $S(F)$ is of full
Lebesgue measure. Recall that the Lebesgue measure is a Gibbs measure relative to the dynamics
$x\mapsto 3x \mod 1$ and that $\int f'(x) d x=0$. By a theorem of Fan-Schmeling (\cite{FS}, Theorem 3.1), the points $x$ such that
$$
     \sum_{k=0}^n f'(3^k x) = O(1)
$$
is of Hausdorff dimension $1$. For every such a point $x$, the series $\sum_{n=0}^\infty a_n f'(3^n x)$
converges. This is checked by the bounded variation hypothesis  $\sum_{n=0}^\infty |a_n -a_{n+1}|<\infty$
and by making an Abel summation by parts. Thus we have proved $\dim_H D(F)=1$, thanks to Proposition \ref{Diff-Cov}.

Suppose $\sum_{n=0}^\infty |a_n|^2 <\infty$. By Theorem \ref{main} and Proposition \ref{Diff-Cov}, we get immediately that $D(F)$ is of full Lebesgue measure. Since $\{f'(3^n x)\}$ is an exact convergence system, $f'$ is not a coboundary. That is to say,
$f'$ can not be written as $u(x) - u(3x)$ for some function $u$. Otherwise $\sum d_n f'(3^n x)$ converges for any sequence $d_n$ decreasing to zero.
Let us consider the Gibbs measure
$\mu_t$ associated to the potential $t f'$ where $t\in \mathbb{R}$ is a parameter.
The measure $\mu_0$ is nothing but the Lebesgue measure.  Let P(t)
be the pressure associated to $tf'$. It is well known that $P$ is a strictly convex and analytic
function on $\mathbb{R}$. Furthermore,
$$
      \int f'(x) d\mu_t(x) = P'(t).
$$
The strict convexity  implies that the mean value $m_t:=\int f' d\mu_t \not =0$ for $t\not= 0$. Let us consider the following series, which is centralized according to $\mu_t$:
$$
   \sum_{n=0}^\infty a_n [f'(3^n x) - m_t].
$$
By Theorem \ref{main}, this series converges $\mu_t$-almost everywhere. As $\sum_{n=0}^\infty a_n$
diverges
and $m_t\not= 0$, the series $\sum_{n=0}^\infty a_n f'(3^n x)$ diverges $\mu_t$-almost everywhere. Thus
$$
    \dim_H S(F) \ge \dim \mu_t,
$$
thanks again to Proposition \ref{Diff-Cov} (see \cite{Fan1994} for the notion of dimension of a measure).
However it is known that $\lim_{t\to 0} \dim \mu_t=1$, so that we can conclude $\dim_H S(F)=1$.
\end{proof}

Remark that the condition that $\sum a_n$ diverges is not needed for the part (a), and the condition $\sum |a_n-a_{n+1}|<\infty$ is not needed for the part (b).

\medskip

Let us finish by examining a concrete example:
$$
   F_\alpha (x) = \sum_{n=1}^\infty n^{-\alpha}3^{-n} e^{2\pi i 3^n x} \qquad (\alpha \in \mathbb{R}).
$$
In the following table, $\lambda$ denote the Lebesgue measure. We have four situations
according to four regions of $\alpha \in \mathbb{R}$:
\medskip

\centerline{
\begin{tabular}{|c|l|}
  \hline
  $\ \ \ \ \ \ \alpha \le 0 $           & $F_\alpha$ is nowhere differentiable  \\
  $0<\alpha\le \frac{1}{2}$ & $\dim D(F_\alpha)=1$, \ \   $\lambda(S(F_\alpha))=1$ \\
  $\frac{1}{2}<\alpha\le 1$ & $\ \, \lambda(D(F_\alpha))=1$, \ \  \ $\dim (F_\alpha)=1$ \\
  $\ \ \ \ \ \ \alpha >1 $              &  $F_\alpha$ is everywhere differentiable \\
  \hline
\end{tabular}
}
\medskip

The classical Weierstrass function $W_\tau(x)=\sum_{n=1}^\infty 3^{- \tau n} \cos (2\pi 3^n x)$ with $0<\tau<1$
is nowhere differentiable. It is recently proved by Bara\'nski, B\'ar\'any and Romanowska \cite{BBR} that the Hausdorff dimension of the
graph of $W_\tau$ is equal to $2-\tau$ for $\tau$ in some interval $(0, \tau_0)$. This is an important step toward the conjecture made by B. Mandelbrot. See \cite{BBR} for information on the subject and for
other discussions on functions of the form $\sum_{n=1}^\infty a_n f(3^nx)$.



\end{document}